\documentclass[12pt,reqno]{amsart}
\usepackage{amscd,amssymb,amsmath,multicol,float,scalefnt,cancel,array,stmaryrd,graphicx}
\restylefloat{table}
\newtheorem{thm}[equation]{Theorem}
\numberwithin{equation}{section}
\newtheorem{cor}[equation]{Corollary}

\newtheorem{rmk}[equation]{Remark}

\newtheorem{lem}[equation]{Lemma}

\newtheorem{prop}[equation]{Proposition}

\newtheorem{tab}[equation]{Table}
\newtheorem{fig}[equation]{Figure}

\begin{document}
\raggedbottom \voffset=-.7truein \hoffset=0truein \vsize=8truein
\hsize=6truein \textheight=8truein \textwidth=6truein
\baselineskip=18truept

\def\mapright#1{\ \smash{\mathop{\longrightarrow}\limits^{#1}}\ }
\def\mapleft#1{\smash{\mathop{\longleftarrow}\limits^{#1}}}
\def\mapup#1{\Big\uparrow\rlap{$\vcenter {\hbox {$#1$}}$}}
\def\mapdown#1{\Big\downarrow\rlap{$\vcenter {\hbox {$\ssize{#1}$}}$}}
\def\mapne#1{\nearrow\rlap{$\vcenter {\hbox {$#1$}}$}}
\def\mapse#1{\searrow\rlap{$\vcenter {\hbox {$\ssize{#1}$}}$}}
\def\mapr#1{\smash{\mathop{\rightarrow}\limits^{#1}}}
\def\ss{\smallskip}
\def\vp{v_1^{-1}\pi}
\def\at{{\widetilde\alpha}}
\def\sm{\wedge}
\def\la{\langle}
\def\ra{\rangle}
\def\on{\operatorname}
\def\ol#1{\overline{#1}{}}
\def\spin{\on{Spin}}
\def\cat{\on{cat}}
\def\lbar{\ell}
\def\qed{\quad\rule{8pt}{8pt}\bigskip}
\def\ssize{\scriptstyle}
\def\a{\alpha}
\def\bz{{\Bbb Z}}
\def\Rhat{\hat{R}}
\def\im{\on{im}}
\def\ct{\widetilde{C}}
\def\ext{\on{Ext}}
\def\sq{\on{Sq}}
\def\eps{\epsilon}
\def\ar#1{\stackrel {#1}{\rightarrow}}
\def\br{{\bold R}}
\def\bC{{\bold C}}
\def\bA{{\bold A}}
\def\bB{{\bold B}}
\def\bD{{\bold D}}
\def\bh{{\bold H}}
\def\bQ{{\bold Q}}
\def\bP{{\bold P}}
\def\bx{{\bold x}}
\def\bo{{\bold{bo}}}
\def\si{\sigma}
\def\Vbar{{\overline V}}
\def\dbar{{\overline d}}
\def\wbar{{\overline w}}
\def\Sum{\sum}
\def\tfrac{\textstyle\frac}
\def\tb{\textstyle\binom}
\def\Si{\Sigma}
\def\w{\wedge}
\def\equ{\begin{equation}}
\def\b{\beta}
\def\G{\Gamma}
\def\L{\Lambda}
\def\g{\gamma}
\def\k{\kappa}
\def\psit{\widetilde{\Psi}}
\def\tht{\widetilde{\Theta}}
\def\psiu{{\underline{\Psi}}}
\def\thu{{\underline{\Theta}}}
\def\aee{A_{\text{ee}}}
\def\aeo{A_{\text{eo}}}
\def\aoo{A_{\text{oo}}}
\def\aoe{A_{\text{oe}}}
\def\vbar{{\overline v}}
\def\endeq{\end{equation}}
\def\sn{S^{2n+1}}
\def\zp{\bold Z_p}
\def\cR{{\mathcal R}}
\def\P{{\mathcal P}}
\def\cQ{{\mathcal Q}}
\def\cj{{\cal J}}
\def\zt{{\bold Z}_2}
\def\bs{{\bold s}}
\def\bof{{\bold f}}
\def\bq{{\bold Q}}
\def\be{{\bold e}}
\def\Hom{\on{Hom}}
\def\ker{\on{ker}}
\def\kot{\widetilde{KO}}
\def\coker{\on{coker}}
\def\da{\downarrow}
\def\colim{\operatornamewithlimits{colim}}
\def\zphat{\bz_2^\wedge}
\def\io{\iota}
\def\Om{\Omega}
\def\Prod{\prod}
\def\e{{\cal E}}
\def\zlt{\Z_{(2)}}
\def\exp{\on{exp}}
\def\abar{{\overline a}}
\def\xbar{{\overline x}}
\def\ybar{{\overline y}}
\def\zbar{{\overline z}}
\def\Rbar{{\overline R}}
\def\nbar{{\overline n}}
\def\cbar{{\overline c}}
\def\qbar{{\overline q}}
\def\bbar{{\overline b}}
\def\et{{\widetilde E}}
\def\ni{\noindent}
\def\coef{\on{coef}}
\def\den{\on{den}}
\def\lcm{\on{l.c.m.}}
\def\vi{v_1^{-1}}
\def\ot{\otimes}
\def\psibar{{\overline\psi}}
\def\thbar{{\overline\theta}}
\def\mhat{{\hat m}}
\def\exc{\on{exc}}
\def\ms{\medskip}
\def\ehat{{\hat e}}
\def\etao{{\eta_{\text{od}}}}
\def\etae{{\eta_{\text{ev}}}}
\def\dirlim{\operatornamewithlimits{dirlim}}
\def\gt{\widetilde{L}}
\def\lt{\widetilde{\lambda}}
\def\st{\widetilde{s}}
\def\ft{\widetilde{f}}
\def\sgd{\on{sgd}}
\def\lfl{\lfloor}
\def\rfl{\rfloor}
\def\ord{\on{ord}}
\def\gd{{\on{gd}}}
\def\rk{{{\on{rk}}_2}}
\def\nbar{{\overline{n}}}
\def\MC{\on{MC}}
\def\lg{{\on{lg}}}
\def\cB{\mathcal{B}}
\def\cS{\mathcal{S}}
\def\cP{\mathcal{P}}
\def\N{{\Bbb N}}
\def\Z{{\Bbb Z}}
\def\Q{{\Bbb Q}}
\def\R{{\Bbb R}}
\def\C{{\Bbb C}}
\def\l{\left}
\def\r{\right}
\def\mo{\on{mod}}
\def\xt{\times}
\def\notimm{\not\subseteq}
\def\Remark{\noindent{\it  Remark}}
\def\kut{\widetilde{KU}}

\def\*#1{\mathbf{#1}}
\def\0{$\*0$}
\def\1{$\*1$}
\def\22{$(\*2,\*2)$}
\def\33{$(\*3,\*3)$}
\def\ss{\smallskip}
\def\ssum{\sum\limits}
\def\dsum{\displaystyle\sum}
\def\la{\langle}
\def\ra{\rangle}
\def\on{\operatorname}
\def\od{\text{od}}
\def\ev{\text{ev}}
\def\o{\on{o}}
\def\U{\on{U}}
\def\lg{\on{lg}}
\def\a{\alpha}
\def\bz{{\Bbb Z}}
\def\eps{\varepsilon}
\def\bc{{\bold C}}
\def\bN{{\bold N}}
\def\nut{\widetilde{\nu}}
\def\tfrac{\textstyle\frac}
\def\b{\beta}
\def\G{\Gamma}
\def\g{\gamma}
\def\zt{{\Bbb Z}_2}
\def\zth{{\bold Z}_2^\wedge}
\def\bs{{\bold s}}
\def\bx{{\bold x}}
\def\bof{{\bold f}}
\def\bq{{\bold Q}}
\def\be{{\bold e}}
\def\lline{\rule{.6in}{.6pt}}
\def\xb{{\overline x}}
\def\xbar{{\overline x}}
\def\ybar{{\overline y}}
\def\zbar{{\overline z}}
\def\ebar{{\overline \be}}
\def\nbar{{\overline n}}
\def\rbar{{\overline r}}
\def\Mbar{{\overline M}}
\def\et{{\widetilde e}}
\def\ni{\noindent}
\def\ms{\medskip}
\def\ehat{{\hat e}}
\def\what{{\widehat w}}
\def\Yhat{{\widehat Y}}
\def\nbar{{\overline{n}}}
\def\minp{\min\nolimits'}
\def\mul{\on{mul}}
\def\N{{\Bbb N}}
\def\Z{{\Bbb Z}}
\def\S{\Sigma}
\def\Q{{\Bbb Q}}
\def\R{{\Bbb R}}
\def\C{{\Bbb C}}
\def\notint{\cancel\cap}
\def\cS{\mathcal S}
\def\cR{\mathcal R}
\def\el{\ell}
\def\TC{\on{TC}}
\def\dstyle{\displaystyle}
\def\ds{\dstyle}

\def\zcl{\on{zcl}}
\def\Vb#1{{\overline{V_{#1}}}}

\def\Remark{\noindent{\it  Remark}}
\title
{$n$-dimensional Klein bottles}
\author{Donald M. Davis}
\address{Department of Mathematics, Lehigh University\\Bethlehem, PA 18015, USA}
\email{dmd1@lehigh.edu}
\date{June 19, 2017}


\maketitle
\begin{abstract}An $n$-dimensional analogue of the Klein bottle arose in our study of topological complexity of planar polygon spaces.
We determine its integral cohomology algebra and stable homotopy type, and give an explicit immersion and embedding in Euclidean space.
 \end{abstract}
\section{Introduction}\label{intro}
The space
\begin{equation}\label{def1}K_n=(S^1)^n/(z_1,\ldots,z_{n-1},z_n)\sim(\zbar_1,\ldots,\zbar_{n-1},-z_n)\end{equation}
arose naturally in the author's study of topological complexity of planar polygon spaces. The model
\begin{equation}\label{def3}K_n\approx ((S^1)^{n-1}\times I)/(z_1,\ldots,z_{n-1},0)\sim(\zbar_1,\ldots,\zbar_{n-1},1)\end{equation}
shows that $K_2$ is the Klein bottle, and $K_n$ is a natural generalization.
Here, of course, $\zbar$ denotes complex conjugation.
A homeomorphism from (\ref{def1}) to (\ref{def3}) is given by $$[(z_1\ldots,z_{n-1},e^{2\pi it_n})]\mapsto [(z_1,\ldots,z_{n-1},2t_n\ \text{mod}\ 1)].$$

The author wrote several papers, culminating in \cite{D2}, computing the topological complexity of the space $\Mbar(\ell)=\Mbar(\ell_1,\ldots,\ell_n)$ of planar polygons with side lengths $\ell_1,\ldots,\ell_n$, identified under isometry. For generic length vectors, this space is an $(n-3)$-manifold, and hence satisfies $\TC(\Mbar(\ell))\le 2n-5$. See \cite{HK} and \cite{F}. Using its mod-2 cohomology algebra, we showed that it is usually true that $\TC(\Mbar(\ell))\ge 2n-6$, within 1 of optimal. In fact, the only planar $n$-gon spaces which are known to have $\TC(\Mbar(\ell))<2n-6$ are those which are homeomorphic to $RP^{n-3}$ (for many values of $n$) or the torus $T^{n-3}$. We feel that planar polygon spaces which are homeomorphic to the spaces $K_{n-3}$ studied here are the best candidates for another such example. We elaborate on this in Section \ref{TCsec}, but have not yet made any advances in this direction.

 In Section \ref{sec3}, we compute $H^*(K_n;\zt)$ as an algebra over the Steenrod algebra, the algebra $H^*(K_n;\Z)$, and $\pi_1(K_n)$.
In Section \ref{sec4}, we  determine the span and immersion and embedding dimensions of these manifolds, and give an explicit immersion of $K_n$ in $\R^{n+1}$, analogous to the familiar picture of a Klein bottle. There is an interesting dependence of the span and embedding dimension of $K_n$ on the parity of $n$.
In Section \ref{shtsec}, we show that $\S K_n$ has the homotopy type of a wedge of spheres and mod-2 Moore spaces.

One might think to consider the related space $K_{n,r}:=(S^1)^n/(z_1,\ldots,z_n)\sim(w_1,\ldots,w_n)$, where $w_i=\begin{cases}\zbar_i&i\le r\\ -z_i&i>r\end{cases}$ for some $r<n-1$. However, this would not be interesting due to the following proposition.
 \begin{prop} The space $K_{n,r}$ defined above is homeomorphic to $K_{r+1}\times (S^1)^{n-r-1}$.\end{prop}
 \begin{proof} There are inverse homeomorphisms $h:K_{n,r}\to K_{r+1}\times (S^1)^{n-r-1}$ and $h':K_{r+1}\times (S^1)^{n-r-1}\to K_{n,r}$ defined by
 $$h[(z_1,\ldots,z_n)]=([(z_1,\ldots,z_{r+1})],(z_{r+1}^{-1}z_{r+2},\ldots,z_{r+1}^{-1}z_n))$$
 and
 $$h'([(z_1,\ldots,z_{r+1})],(z_{r+2},\ldots,z_n))=[(z_1,\ldots,z_{r+1},z_{r+1}z_{r+2},\ldots,z_{r+1}z_n)].$$
 \end{proof}

\section{Cohomology and fundamental group of $K_n$}\label{sec3}
We begin by determining $H^*(K_n;\zt)$ as an algebra. The following lemma is useful. We thank J.-C.Hausmann for discussions about this lemma.
\begin{lem} \label{lem1}Suppose $M$ is a space with free involution $\tau$, with quotient $\Mbar$. Let $X=(S^1\times M)/(z,x)\sim(\zbar,\tau(x))$. There is an algebra isomorphism
$$H^*(X;\zt)\approx H^*(\Mbar;\zt)[y]/(y^2=w_1y),$$
where $|y|=1$ and $w_1\in H^1(\Mbar;\zt)$ classifies the double cover $M\to\Mbar$.\end{lem}
\begin{proof} The space $X$ is the sphere bundle of the 2-plane bundle $\theta$ over $\Mbar$ given by
$$(\R\times \R\times M)/(s,t,x)\sim(s,-t,\tau(x))\to \Mbar.$$
There is a cofiber sequence $S(\theta)\mapright{j}D(\theta)\to T(\theta)$ with $S(\theta)=X$, $D(\theta)\simeq \Mbar$, and a section $s:\Mbar\to X$ defined by $s([x])=[(1,0,x)]$. Thus there is a split SES
$$0\to H^*(\Mbar)\mapright{j^*}H^*(X)\to H^{*+1}(T\theta)\to 0,$$
with all coefficients in this proof being $\zt$. Let $y\in H^1(X)$ correspond to the Thom class $U\in H^2(T\theta)$ under this splitting. Then $y^2=\sq^1y$ corresponds to $\sq^1U=w_1(\theta)\cup U$. Since, as a real bundle, $\theta$ is isomorphic to the sum of a trivial bundle and the line bundle associated to the double cover, we obtain $y^2=w_1y$. Using the Thom isomorphism, we obtain the ring isomorphism
$$H^*X\approx H^*\Mbar\oplus H^*\Mbar\cdot y$$
with $y^2=w_1y$.\end{proof}

The following result was obtained in a much different form and by much different methods in \cite{HK}.
\begin{thm} There is an algebra isomorphism
$$H^*(K_n;\zt)\approx\zt[R,V_1,\ldots,V_{n-1}]/(R^2,V_i^2+RV_i),$$
with $|R|=|V_i|=1$.\label{cohthm}
\end{thm}
\begin{proof} This follows by induction on $n$ from Lemma \ref{lem1} with $M=(S^1)^{n-1}$, $\Mbar=K_{n-1}$, and $X=K_n$. The only additional ingredient required is to know that the class $R$ which classifies the double cover $(S^1)^{n-1}\to K_{n-1}$ pulls back to the similar class for $(S^1)^n\to K_n$. This follows from the fact that there is a pullback diagram
$$\begin{CD} (S^1)^n@>p>> (S^1)^{n-1}\\
@VVq_nV @VVq_{n-1}V\\
K_n@>\overline{p}>> K_{n-1}.\end{CD}$$
Here $p(w_1,\ldots,w_n)=(w_1,\ldots,w_{n-2},w_n)$, and $\overline{p}$ is defined similarly.
This pullback property is proved by noting that a  point in the pullback is
$$([z_1,\ldots,z_n],(w_1,\ldots,w_{n-2},w_n)),$$
where $[z_1,\ldots,z_n]=[\zbar_1,\ldots,\zbar_{n-1},-z_n]$, such that
$$(z_1,\ldots,z_{n-2},z_n)=(w_1,\ldots,w_{n-2},w_n)\text{ or }(\wbar_1,\ldots,\wbar_{n-2},-w_n).$$
Such a point is uniquely described as $(z_1,\ldots,z_n)$ if $z_n=w_n$, or $(\zbar_1,\ldots,\zbar_{n-2},-z_n)$ if $z_n=-w_n$.
\end{proof}

One corollary is the precise value of the (reduced) Lusternik-Schnirelmann category $\cat(K_n)$.
\begin{cor} For $n\ge2$, we have $\cat(K_n)=n$.\end{cor}
\begin{proof} Since $H^*(K_n;\zt)$ has an $n$-fold nontrivial cup product, $n\le\cat(K_n)$ by \cite[Prop 1.5]{LS}, and since $K_n$ is an $n$-manifold, $\cat(K_n)\le n$ by \cite[Thm 1.7]{LS}.\end{proof}

Also, we have the following immediate corollary, the entire $A$-module structure.
\begin{cor} In $H^*(K_n;\zt)$, for $j>0$ and distinct subscripts of $V$,
$$\sq^j(R^\eps V_{i_1}\cdots V_{i_r})=\begin{cases} RV_{i_1}\cdots V_{i_r}&j=1,\ \eps=0,\ r\text{ odd}\\
0&\text{otherwise.}\end{cases}$$\label{sqcor}
\end{cor}
\begin{proof} Since $R^2=0$, $\sq^1(R^\eps V_{i_1}\cdots V_{i_r}) =rR^{\eps+1}V_{i_1}\cdots V_{i_r}$. Action of $\sq^j$ for $j>1$ will have all terms divisible by $R^2=0$.\end{proof}

This result suggested the possibility of a splitting of $\S K_n$, which we prove in Theorem \ref{split}.

The depiction of $H^*(K_4;\zt)$ in Table \ref{K4} might be helpful. The horizontal lines indicate the action of $\sq^1$.

\begin{tab}\label{K4}

\begin{center}
\begin{tabular}{crl}

$H^0$&$H^1\ \quad H^2$&$\quad H^3\quad\qquad H^4$\\
\hline
$1$&$R\quad\qquad$&\\
&$V_1-RV_1$&\\
&$V_2-RV_2$&\\
&$V_3-RV_3$&\\
&$V_1V_2$&$RV_1V_2$\\
&$V_1V_3$&$RV_1V_3$\\
&$V_2V_3$&$RV_2V_3$\\
&&$V_1V_2V_3-RV_1V_2V_3$

\end{tabular}
\end{center}
\end{tab}

Theorem \ref{cohthm} is used later to prove Proposition \ref{zclprop}, which is used for lower bounds of topological complexity.

Denote by $\L_\cR$ the exterior algebra over a ring $\cR$ on a set of generators, with superscript $\od$ (resp.~$\ev$) (resp.~$k$) referring to the subspace spanned by products of an odd (resp.~even) number of (resp.~$k$) generators.
\begin{thm}\label{cohalg} There are elements $R$ and $Z_i$ of grading 1 such that there is an isomorphism of graded rings
$$H^*(K_n;\Z)\approx \L_\Z^\ev[Z_1,\ldots,Z_{n-1}]\oplus R\cdot\L_\Z^\ev[Z_1,\ldots,Z_{n-1}]\oplus R\cdot\L_{\zt}^\od[Z_1,\ldots,Z_{n-1}],$$
with $R^2=0$ and products of elements in the first summand with all others as in the exterior algebra.
\end{thm}
\begin{proof}
We use the description of $K_n$ in (\ref{def3}). If $h:X\to X$ is a homeomorphism and
$\widetilde X=X\times I/(x,0)\sim(h(x),1)$, then a Mayer-Vietoris argument shows that there is an exact sequence, with any coefficients,
$$\to H^r(\widetilde X)\to H^r(X) \mapright{h^*-1} H^r(X)\to H^{r+1}(\widetilde X)\to.$$
This can be obtained by letting $A=X\times(0,1)$ and $B=\widetilde X-(X\times\{\frac12\})$. Then $H^*(A)\oplus H^*(B)\to H^*(A\cap B)$ becomes
$$H^*(X)\oplus H^*(X)\to H^*(X)\oplus H^*(X)$$
with connecting homomorphism $\bigl(\begin{smallmatrix} 1&1\\1&h^*\end{smallmatrix}\bigr)$, of which the kernel and cokernel are the same as that of $h^*-1$ on $H^*(X)$.

In our case, $X=T^{n-1}$ and $(h^*-1)$ on $H^r(T^{n-1})$ is multiplication by $(-1)^r-1$.
We obtain commutative diagrams of exact sequences in which $X_i\mapsto V_i$.
$$\begin{CD}0@>>> \L^{2k-1}_{\zt}[X_1,\ldots,X_{n-1}]@>\delta>> H^{2k}(K_n;\Z)@>q>>\L^{2k}_\Z[X_1,\ldots,X_{n-1}]@>>>0\\
@. @V\rho_2VV @V\rho VV @VVV @.\\
0@>>> \L^{2k-1}_{\zt}[V_1,\ldots,V_{n-1}]@>\delta'>> H^{2k}(K_n;\zt)@>>>\L^{2k}_{\zt}[V_1,\ldots,V_{n-1}]@>>>0\end{CD}$$
\medskip

$$\begin{CD}0@>>> \L^{2k}_{\Z}[X_1,\ldots,X_{n-1}]@>\delta>> H^{2k+1}(K_n;\Z)@>>>0\\
@. @VVV @V\rho VV @VVV \\
0@>>> \L^{2k}_{\zt}[V_1,\ldots,V_{n-1}]@>\delta'>> H^{2k+1}(K_n;\zt)@>>>\L^{2k+1}_{\zt}[V_1,\ldots,V_{n-1}]@>>>0\end{CD}$$
The homomorphisms $\delta'$ are multiplication by $R$ in Theorem \ref{cohthm}.

The exact sequences show clearly that the abelian group structure of $H^*(K_n;\Z)$ is as claimed. Some care is required to show that the product structure is, too.

If $S=\{s_1,\ldots,s_\ell\}\subset[\![n-1]\!]=\{1,\ldots,n-1\}$, there is a natural map $p_{S,n}:K_n\to K_{\ell+1}$ sending $[(z_1,\ldots,z_{n-1},t)]\mapsto [(z_{s_1},\ldots,z_{s_\ell},t)]$. The induced cohomology homorphisms are compatible with the above diagrams, and are injective.

For all $m$, $H^1(K_m;\Z)=\Z$, generated compatibly by $R=\delta(1)$ in the second diagram. If $m$ is odd, $K_m$ is orientable by Proposition \ref{tang}. Let $Z_{[\![m-1]\!]}\in H^{m-1}(K_m;\Z)$ denote the cap product of  an orientation class with $R$ which satisfies $q(Z_{[\![m-1]\!]})=X_1\cdots X_{m-1}$ in the first diagram. Thus our orientation class is $R\cdot Z_{[\![m-1]\!]}$.

For $S=\{s_1,\ldots,s_{2k}\}\subset [\![n-1]\!]$ and $\epsilon\in\{0,1\}$, let $R^\epsilon Z_S\in H^{2k+\epsilon}(K_n;\Z)$ equal $p_{S,n}^*(R^\epsilon Z_{[\![2k]\!]})$. This class is what we will call
$R^\epsilon Z_{s_1}\cdots Z_{s_{2k}}$, once we establish the multiplicative structure. Note that single classes $Z_s$ do not exist.

These classes satisfy the multiplicative structure of an exterior algebra (e.g., $Z_{i,j}Z_{k,\ell}=-Z_{i,k}Z_{j,\ell}$ if $i<j<k<\ell$) since they do when $q$ or $\rho$ is applied in the first diagram, and $\rho_2$ is bijective. Thus we rename them as $R^\epsilon Z_{s_1}\cdots Z_{s_{2k}}$; they comprise the first two summands in the statement of the theorem.

The remaining classes are in $\im(\delta)$ in the first diagram. Since these classes have order 2, the product formulas involving them and (perhaps) the $\Z$ classes above are implied by Theorem \ref{cohthm}.
\end{proof}

We can use a combination of the Atiyah-Hirzebruch spectral sequence and the exact sequences used in the cohomology proof above to obtain similar results for the ring structure of $KU^*(K_n)$ and $KO^*(K_n)$, but the results are not particularly surprising or useful. Theorem \ref{split} is also helpful. For example, the ring $KO^*(K_4)$ is isomorphic to
\begin{eqnarray*}&&KO^*\la R,X_1X_2,X_1X_3,X_2X_3,RX_1X_2,RX_1X_3,RX_2X_3\ra\\
&\oplus&KO^*(M^0(2))\la RX_1,RX_2,RX_3,RX_1X_2X_3\ra,\end{eqnarray*}
where $R$ and $X_i$ have grading 1.

The fundamental group of $K_n$ is a straightforward generalization of that of the Klein bottle.
 \begin{prop} The fundamental group $\pi_1(K_n)$ has generators $a_1,\ldots,a_n$ with relations $a_ja_n=a_na_j^{-1}$, $1\le j\le n-1$, and $a_ia_j=a_ja_i$, $1\le i<j\le n-1$. The double cover $p:T^n\to K_n$ satisfies $p_*(g_j)=\begin{cases}a_j&j<n\\ a_n^2&j=n.\end{cases}$.\end{prop}
 \begin{proof} Using model (\ref{def1}) for $K_n$, let $a_j=[f_j]$, where $f_j:I\to K_n$ is defined by
 $$f_j(t)=\begin{cases}[(1^{j-1},e^{2\pi i t},1^{n-j})]&j<n\\
 [(1^{n-1},e^{\pi it})]&j=n.\end{cases}$$
 The homotopy between $f_jf_n$ and $f_n\overline{f_j}$ is exactly as in the Klein bottle, and the commuting of $a_i$ and $a_j$ follows from that in the torus. Since $p_*$ is an isomorphism from $\pi_1(T^n)$ to an index-2 subgroup of the group described, this group must equal $\pi_1(K_n)$.\end{proof}

\section{Span, immersions, and embeddings of $K_n$}\label{sec4}
In this section, we show that if $n$ is odd, $K_n$ is parallelizable and embeds in $\R^{n+1}$, while if $n$ is even, it has $n-1$ linearly independent vector fields. For all $n$, we obtain an explicit immersion of $K_n$ in $\R^{n+1}$ and embedding in $\R^{n+2}$, analogous to the familiar picture of a Klein bottle.

We begin with the following result for the tangent bundle.
\begin{prop}\label{tang} For $k>0$, the Stiefel-Whitney classes of the tangent bundle of $K_n$ are given by
$$w_k(\tau(K_n))=\begin{cases}R&k=1,\ n\text{ even}\\ 0&\text{otherwise.}\end{cases}$$
\end{prop}
\begin{proof} We use Wu's formula, as given in \cite[Thm 11.4]{MS}, which states that, for an $n$-manifold $M$, if $v_j$ denotes the $j$th Wu class, which satisfies $v_j\cup x=\sq^jx$ for all $x\in H^{n-j}(M)$, then $w_k(\tau(M))=\ds\sum_j\sq^{k-j}v_j$. Since, using Corollary \ref{sqcor}, for $j>0$, $v_j=0$ in $H^*(K_n)$ unless $j=1$ and $n$ is even, in which case $v_1=R$, the result follows since only $\sq^0$ acts nontrivially on $R$.\end{proof}

This leads us to the following stronger result.
\begin{thm}\label{tangbdl} If $n$ is odd, the tangent bundle $\tau(K_n)$ is isomorphic to a trivial bundle. If $n$ is even, $\tau(K_n)\approx\eta\oplus(n-1)\eps$,  where $\eta$ is a line bundle with $w_1=R$, and $2\tau(K_n)$ is isomorphic to a trivial bundle.\end{thm}
\begin{proof} Using model (\ref{def3}), $\tau(K_n)$ has total space
$$(\R^n\times (S^1)^{n-1}\times I)/(t_1,\ldots,t_n,z_1,\ldots,z_{n-1},0)\sim(-t_1,\ldots,-t_{n-1},t_n,\zbar_1,\ldots,\zbar_{n-1},1).$$
If $n$ is odd, an isomorphism to the trivial bundle,
$$(\R^n\times (S^1)^{n-1}\times I)/(t_1,\ldots,t_n,z_1,\ldots,z_{n-1},0)\sim(t_1,\ldots,t_{n-1},t_n,\zbar_1,\ldots,\zbar_{n-1},1),$$
is given by sending
$$\bigl(\tbinom{t_1}{t_2},\ldots,\tbinom{t_{n-2}}{t_{n-1}},t_n,z_1,\ldots,z_{n-1},s\bigr)$$
to
$$\bigl(\bigl(\begin{smallmatrix}\cos(\pi s)&-\sin(\pi s)\\ \sin(\pi s)&\cos(\pi s)\end{smallmatrix}\bigr)\tbinom{t_1}{t_2},\ldots,\bigl(\begin{smallmatrix}\cos(\pi s)&-\sin(\pi s)\\ \sin(\pi s)&\cos(\pi s)\end{smallmatrix}\bigr)\tbinom{t_{n-2}}{t_{n-1}},t_n,z_1,\ldots,z_{n-1},s\bigr).$$
If $n$ is even, there is a similar isomorphism from $\tau(K_n)$ to $(n-1)\eps\oplus\eta$, where $\eta$ has total space
$$(\R\times (S^1)^{n-1}\times I)/(t,z_1,\ldots,z_{n-1},0)\sim(-t,\zbar_1,\ldots,\zbar_{n-1},1),$$
with the $t$ corresponding to $t_{n-1}$. This is the line bundle  associated to the double cover, with $w_1=R$. Twice this bundle is trivial, using the same rotation matrices as above.\end{proof}

We quickly deduce the span and immersion dimension of $K_n$.
\begin{cor} The span of $K_n$ (i.e., maximal number of linearly independent vector fields) is $n$ if $n$ is odd, and $n-1$ if $n$ is even. For all $n$, $K_n$ immerses in $\R^{n+1}$.
\end{cor}
\begin{proof} Since the span is the dimension of the largest trivial subbundle, that part is immediate from Theorem \ref{tangbdl}. By Hirsch's Theorem (\cite{Hir}), $K_n$ immerses in $\R^{n+1}$ since there is a 1-dimensional vector bundle over it, $\eps$ if $n$ is odd and $\eta$ if $n$ is even, whose sum with the tangent bundle is trivial.\end{proof}

We can obtain an explicit immersion of $K_n$ in $\R^{n+1}$, analogous to the familiar picture of the Klein bottle $K_2$ in $\R^3$.  We use the following lemma, presumably well known.
\begin{lem}\label{emb} Let $\thbar=(\theta_1,\ldots,\theta_{n-1})$ with $\theta_i\in\R\mod 2\pi$. A parametrization $(x_1(\thbar),\ldots,x_{n}(\thbar))$ of an embedding of $T^{n-1}$ in $\R^{n}$ satisfying $x_1(-\thbar)=x_1(\thbar)$ and $x_i(-\thbar)=-x_i(\thbar)$ for $2\le i\le n$ can be given as follows. Choose positive real numbers $r_i$, $1\le i\le n-1$, satisfying $r_i>\ds\sum_{j>i}r_j$.
Let
\begin{eqnarray*}w_{n}&=&r_{n-1}\\
w_i&=&r_{i-1}+w_{i+1}\cos(\theta_i)\text{ for }1<i\le n-1\\
x_i&=&w_i\sin(\theta_{i-1})\text{ for }1<i\le n\\
x_1&=&w_2\cos(\theta_1).\end{eqnarray*}\end{lem}
For example, if $n=4$,
\begin{eqnarray*}
x_1&=&(r_1+(r_2+r_3\cos(\theta_3))\cos(\theta_2))\cos(\theta_1)\\
x_2&=&(r_1+(r_2+r_3\cos(\theta_3))\cos(\theta_2))\sin(\theta_1)\\
x_3&=&(r_2+r_3\cos(\theta_3))\sin(\theta_2)\\
x_4&=&r_3\sin(\theta_3).\end{eqnarray*}

\begin{proof}[Proof of Lemma \ref{emb}] The proof is by induction on $n$. Assume known for $n-1$. Take the parametrized $T^{n-2}$, using $r$-values $r_2,\ldots,r_{n-1}$ and  $\theta$ values $\theta_2,\ldots,\theta_{n-1}$, in the $x_1x_3\cdots x_{n}$-plane. Translate it by $r_1$ units in the $x_1$ coordinate. Rotate it around the $x_3\cdots x_{n }$-plane. All $x_i$ with $i\ge3$ remain unaffected, while
\begin{eqnarray*}x_1&=&(r_1+x_1')\cos(\theta_1)\\
x_2&=&(r_1+x_1')\sin(\theta_1),\end{eqnarray*}
where $x_1'$ is the $x_1$-value before translating.
\end{proof}
Note that the maximum $x$ in the embedding of Lemma \ref{emb} is $r_1+\cdots+r_{n-1}$.

\begin{rmk}\label{rmk}{\rm By varying $r_{n-1}$ through an appropriate range of values, we can obtain a family of smoothly embedded disjoint $T^{n-1}$'s in $\R^n$. For example, if $D\le r_{n-1}\le2D$, and $r_i=2^{n-i}D$ for $i<n-1$ we have disjointly embedded $T^{n-1}$'s with maximum $x$ ranging from $(2^n-3)D$ to $(2^n-2)D$.}\end{rmk}

Next we review the parametrization of the Klein bottle given in \cite{Fr}. This uses the curve
$$\a(t)=\la 5\sin(t),2\sin^2(t)\cos(t),0\ra,\qquad0\le t\le\pi$$ as directrix. Note that $\a(\pi)=\a(0)$ and $\a'(\pi)=-\a'(0)$. This curve passes through the center of the band in Figure \ref{band}. Orthogonal to the directrix are unit vectors $J(t)=\la-v_2,v_1,0\ra$ if $\a'(t)/\|\a'(t)\|=\la v_1,v_2,0\ra$. Note that $J(\pi)=-J(0)$. With $J(t)=\la j_1,j_2,0\ra$ and
\begin{equation}\label{r}r(t)=\tfrac12-\tfrac1{15}(2t-\pi)\sqrt{t(\pi-t)},\end{equation}
the immersed Klein bottle is parametrized by
$$k(\theta,t)=\a(t)+r(t)\la j_1\cos(\theta),j_2\cos(\theta),\sin(\theta)\ra,\qquad0\le \theta\le2\pi,\ 0\le t\le \pi.$$
In Figure \ref{band}, we illustrate the projection onto the $x_1x_2$-plane. There are circles of radius $r(t)$ perpendicular to the $x_1x_2$-plane, with each of the indicated lines as diameters.

\medskip
\begin{minipage}[h]{6in}
\begin{fig}
\label{band}

\begin{center}
\includegraphics[width=4in,height=3in]{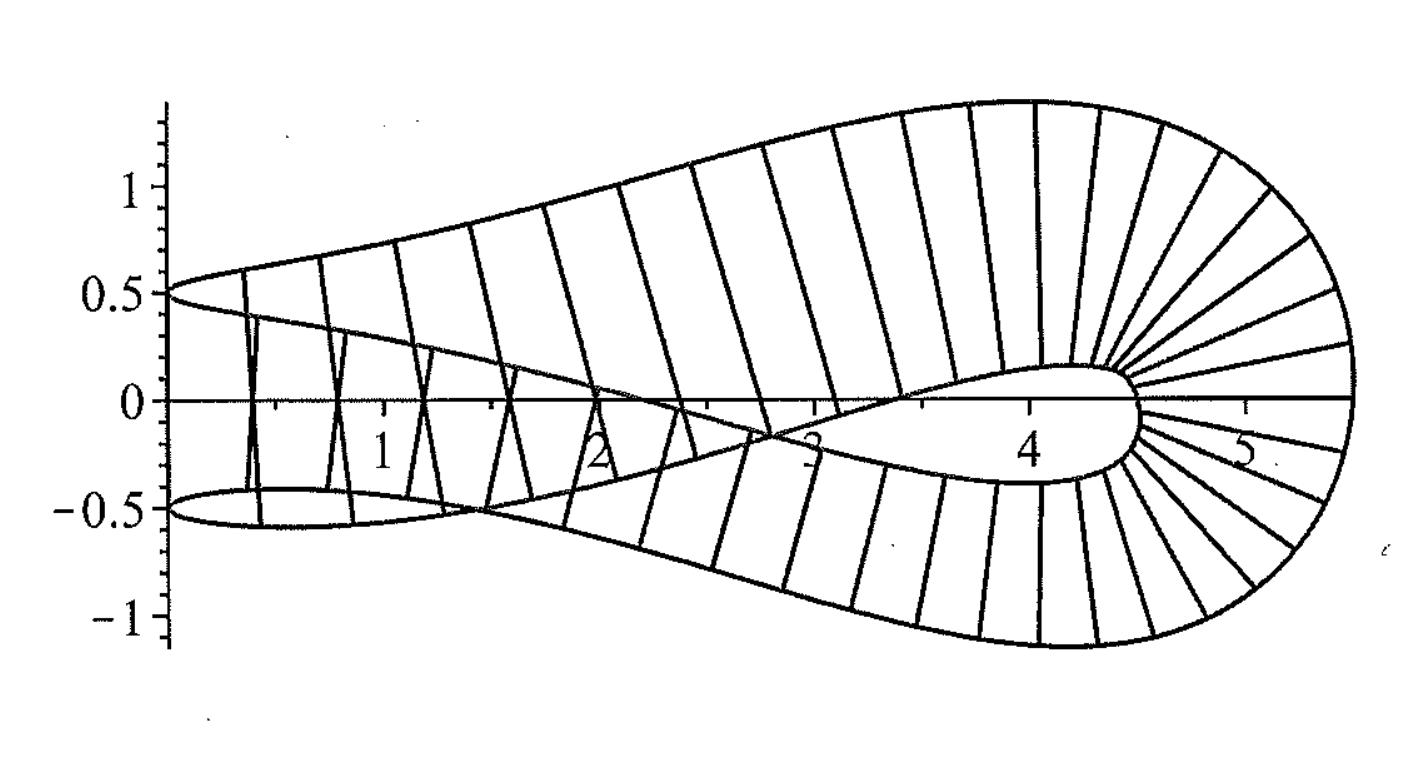}
\end{center}
\end{fig}
\end{minipage}
\medskip

We will make a similar immersion of $K_n$ in $\R^{n+1}$ by placing disjoint $T^{n-1}$'s above lines similar to those in Figure \ref{band}. As $t$ varies from $\pi-\epsilon$ to $\pi$, and then from 0 to $\epsilon$, the values $r(t)$ will be varying, and we wish the associated $T^{n-1}$'s in the $\R^n$ sitting above the appropriate segments to be disjoint. To this end, we must change the formula (\ref{r}) slightly.  If the $\frac1{15}$ in (\ref{r}) is replaced by a number $d$, then a calculus exercise shows that $\frac12-\frac{\pi^2}4 d\le r(t)\le\frac12+\frac{\pi^2}4 d$  for all $t$. By choosing $d=2/(\pi^2(2^{n+1}-5))$, this interval of $r(t)$ values that we will encounter has the property that if we choose the values $r_i$ which determine an embedding of $T^{n-1}$ in $\R^n$ as in Remark \ref{rmk} with $D=1/(2^{n+1}-5)$, then distinct values of $r(t)$ will have disjointly embedded $T^{n-1}$'s with maximum $x$ equal to $r(t)$.

For $\frac12-\frac1{2(2^{n+1}-5)}\le s\le \frac12+\frac1{2(2^{n+1}-5)}$, let $(x_1(s,\thbar),\ldots,x_{n}(s,\thbar))$ be the embedding of $T^{n-1}$ in $\R^{n}$  in Lemma \ref{emb} with $r_{n-1}=s-\frac12+\frac3{2(2^{n+1}-5)}$ and $r_i=2^{n-i}/(2^{n+1}-5)$ for $1\le i\le n-2$. Now with $\a(t)=\la5\sin(t),2\sin^2(t)\cos(t),0\ldots,0\ra$,  $j_1(t)$ and $j_2(t)$ as above, and $$r(t)=\frac12-\frac2{\pi^2(2^{n+1}-5)}(2t-\pi)\sqrt{t(\pi-t)},$$ our parametrization of an immersion of $K_{n}$ in $\R^{n+1}$ is given by
\begin{equation}\label{immeq}k_{n}(\thbar,t)=\a(t)+\la x_1(r(t),\thbar)j_1(t),x_1(r(t),\thbar)j_2(t),x_2(r(t),\thbar),\ldots,x_{n}(r(t),\thbar)\ra,\end{equation}
for $0\le\theta_i\le 2\pi,\ 0\le t\le\pi$.
We have $k_{n}(\thbar,0)=-k_{n}(-\thbar,\pi)$, which makes it a model of $K_{n}$. The $[0,\pi]$ that we use for $t$ here corresponds to the $[0,1]$ in (\ref{def3}); we use $[0.\pi]$ primarily for consistency with \cite{Fr}.

To see that this is locally an embedding, note that $k_n(T^{n-1}\times t)\subset J_t\times \R^{n-1}$, where $J_t\subset\R^2$ is the segment with endpoints $\la5\sin t,2\sin^2t\cos t\ra\pm\la r(t)j_1(t),r(t)j_2(t)\ra$.
For nearby values of $t\in(0,\pi)$, the segments $J_t$ are disjoint, and $T^{n-1}$ is embedded in the $n$-dimensional space $J_t\times\R^{n-1}$. For a small positive $t$ and a $t$ just less than $\pi$, the segments $J_t$ are not disjoint, but the values of $r(t)$ vary in this small neighborhood of $t=0$, and so the $T^{n-1}$'s are disjoint, due to our choice of the values of $r_i$.

We can expand (\ref{immeq}) to an explicit embedding of $K_n$ in $\R^{n+2}$ by
\begin{equation}\label{expemb}f(\thbar,t)=\la k_{n}(\thbar,t),\sin(2t)\ra.\end{equation}
This is an embedding since the only points where $k_n(\thbar,t)=k_n(\thbar',t')$ have $\sin(2t)$ and $\sin(2t')$ with opposite signs.

We improve this when $n$ is odd in the following result, which benefited from a discussion with Ryan Budney.
\begin{thm} \label{embthm} If $n$ is odd, then $K_n$ can be embedded in $\R^{n+1}$.\end{thm}
\begin{proof} We observe that $K_n$ is the total space of the $S^1$-bundle over $K_{n-1}$ associated to the 2-plane bundle $\eta\oplus\eps$, where $w_1(\eta)=R$. Since $n-1$ is even, the tangent bundle $\tau_{n-1}$ has $w_1=R$ by Proposition \ref{tang}. Thus so does the tubular neighborhood $\zeta$ of the immersion of $K_{n-1}$ in $\R^n$ constructed above, when interpreted as a line bundle, since $\tau_{n-1}\oplus\zeta$ is trivial. Since line bundles are classified by $w_1$, we deduce that $\eta$ and $\zeta$ are isomorphic. If the immersion is expanded to an embedding in $\R^{n+1}$ as in (\ref{expemb}), the tubular neighborhood is just $\zeta\oplus\eps$, as it just adds a component in the new direction. Let $h$ be a bundle isomorphism of the disk bundle $D$ of $\eta\oplus\eps$ to that of $\zeta\oplus\eps$, interpreted as a subset of $\R^{n+1}$. The restriction of $h$ to the boundary of $D$ is our desired embedding of $K_n$ in $\R^{n+1}$.
\end{proof}

If $n$ is even, $K_n$ cannot be embedded in $\R^{n+1}$ since $w_1\ne0$. Combining these observations with the above theorem, we have the following corollary.
\begin{cor} The smallest Euclidean space in which $K_n$ can be embedded is $\R^{n+1}$ if n is odd, and $\R^{n+2}$ if $n$ is even.\end{cor}

\section{Splitting of $\Sigma K_n$}\label{shtsec}
In this section, we obtain an explicit splitting of $\Sigma K_n$ as a wedge of spheres and mod-2 Moore spaces.
 Throughout, $SX$ denotes unreduced suspension, and $\Sigma X$ reduced suspension. Cones are always reduced.
 We begin with a lemma.
\begin{lem}\label{h} Let $h$ be a self-homeomorphism of a pointed space $W$, and $$Z=(W\times I)/(w,0)\sim(h(w),1),$$ where the relation applies to all $w\in W$. Then
$$\Sigma Z\simeq S^2\vee \MC(-\Sigma h\vee1:\Sigma W\to\Sigma W),$$ where this latter map refers to the composite $$\Sigma W\mapright{p}\Sigma W\vee \Sigma W\mapright{-\Sigma h\vee1}\Sigma W,$$ where $(-\S h)([t,w])=[1-t,h(w)]$.
\end{lem}
\begin{proof} We consider the cofiber sequence induced by the map $W\to Z$ defined by $w\mapsto[w,\frac12]$, so $\Sigma Z$ has the homotopy type of the mapping cone of the collapse map $Z\cup CW\mapright{c}\Sigma W$. We precede $c$ by a homotopy equivalence $Z/W\mapright{j} Z\cup CW$. This map
$$\frac{W\times I}{(w,0)\sim(h(w),1),W\times\frac12}\mapright{j}\frac{W\times I\cup C(W\times\frac12)}{(w,0)\sim(h(w),1)}$$
can be defined by $$j([w,t])=\begin{cases}[w,2t]&0\le t\le\frac14\\
[4t-1,w,\frac12]&\frac14\le t\le\frac12\\
[3-4t,w,\frac12]&\frac12\le t\le\frac34\\
[w,2t-1]&\frac34\le t\le 1.\end{cases}$$
The composite $c\circ j:\dfrac{W\times I}{(w,0)\sim(h(w),1),W\times\frac12}\to \Sigma W$ sends
$$[w,t]\mapsto\begin{cases}*&0\le t\le\frac14\\
[4t-1,w]&\frac14\le t\le\frac12\\
[3-4t,w]&\frac12\le t\le\frac34\\
*&\frac34\le t\le 1,\end{cases}$$
and this is homotopic to
$$j':[w,t]\mapsto\begin{cases}[2t,w]&0\le t\le \frac12\\
[2-2t,w]&\frac12\le t\le1.\end{cases}$$

 There is a homotopy equivalence $k:S^1\vee SW\to \dfrac{W\times I}{(w,0)\sim(h(w),1),W\times\frac12}$ defined by
 $$[t]\mapsto\begin{cases}[w_0,t+\frac12]&0\le t\le \frac12\\ [w_0,t-\frac12]&\frac12\le t\le1,\end{cases}\qquad[t,w]\mapsto\begin{cases}[h(w),t+\frac12]&0\le t\le\frac12\\ [w,t-\frac12]&\frac12\le t\le1.\end{cases}$$
To see that $k$ is a homotopy equivalence, write it (up to a slight reparametrization of $S^1$) as
$$S^1\vee SW\mapright{k'}\frac{W\times[0,\frac12]\cup I\cup W\times[\frac12',1]}{(w,0)\sim(hw,1),(w,\frac12)\sim0,1\sim(w,\frac12')}\mapright{c'}\frac{W\times I}{(w,0)\sim(h(w),1),W\times\frac12},$$
where $k'([t,w])=k([t,w])$, $c'$ collapses $I$, and
$$k'([t])=\begin{cases}[2t]&0\le t\le\frac12\\
[w_0,2t-\frac12]&\frac12\le t\le\frac34\\
[w_0,2t-\frac32]&\frac34\le t\le1.\end{cases}$$
This $k'$ is a homotopy equivalence by Whitehead's Theorem, as VanKampen/Mayer-Vietoris imply that it  induces an isomorphism in $\pi_1$ and $H_*$, and $c'$ is a homotopy equivalence since it collapses a contractible subspace.

The composite $j'\circ k:S^1\vee SW\to\Sigma W$ sends $S^1$ to the basepoint and $[t,w]$ to $\begin{cases}[1-2t,h(w)]&0\le t\le\frac12\\ [2t-1,w]&\frac12\le t\le 1.\end{cases}$
Thus it factors through $c_{w_0}\vee(-\Sigma h\vee1):S^1\vee \Sigma W\to \Sigma W$. Hence
$$\Sigma Z\simeq \MC(c\circ j)\simeq\MC(j')\simeq\MC(j'\circ k)\simeq S^2\vee\MC(-\Sigma h\vee1).$$
\end{proof}

We easily deduce the following corollary.
\begin{cor}\label{fg} If $f:X\to X$ and $g:Y\to Y$ are homeomorphisms of pointed spaces, and $Z=(X\times Y\times I)/(x,y,0)\sim(f(x),g(y),1)$, then
$$\Sigma Z\simeq S^2\vee C_X\vee C_Y\vee C_{X\w Y},$$
where
\begin{eqnarray*}C_X&=&\MC(-\Sigma f\vee1:\S X\to \S X)\\
C_Y&=&\MC(-\S g\vee1:\S Y\to \S Y)\\
C_{X\w Y}&=&\MC(-\S f\w g\vee1:\S X\w Y\to\S X\w Y).\end{eqnarray*}
\end{cor}
\begin{proof} We apply Lemma \ref{h} to $f\times g:X\times Y\to X\times Y$, using the following commutative diagram, in which vertical maps are homotopy equivalences.
$$\begin{CD}\S(X\times Y)@>p>>\S(X\times Y)\vee\S(X\times Y)@>-\S(f\times g)\vee1>>\S(X\times Y)\\
@VV qV @VV q\vee q V @VV qV\\
\S X\vee\S Y\vee\S X\w Y @>p>> B @>b>> \S X\vee\S Y\vee \S X\w Y,\end{CD}$$
where $B=\S X\vee\S X\vee\S Y\vee\S Y\vee(\S X\w Y)\vee(\S X\w Y)$ and $$b=(-\S f\vee1)\vee(-\S g\vee1)\vee(-\S f\w g\vee1).$$ By Lemma \ref{h}, $\Sigma Z$ has the homotopy type of the one-point union of $S^2$ with the cofiber of the first horizontal composite, which has the homotopy type of the cofiber of the second horizontal composite, and this is what is claimed in this corollary.\end{proof}

This corollary can, of course, be iterated to a product of many spaces. Our space $K_{n}$ is the space $Z$ in the iteration of the corollary applied to the map $T^{n-1}\mapright{h} T^{n-1}$ which is the reflection map $z\mapsto\zbar$ in each factor. Note that
$$\S T^{n-1}\simeq \bigvee_{R\subset[\![n-1]\!]}\S S^{|R|},$$
where $R$ ranges over all nonempty subsets of $[\![n-1]\!]=\{1,\ldots,n-1\}$. The map $-\S h\vee1:\S S^{|R|}\to\S S^{|R|}$ has degree $1-(-1)^{|R|}$. The corollary says that our $\S K_{n}$ has the homotopy type of
$$S^2\vee\bigvee_{i=1}^{n-1}\tbinom {n-1}i\MC(S^{i+1}\mapright{1-(-1)^i}S^{i+1}),$$
from which the following splitting result follows immediately.
\begin{thm}\label{split} There is a homotopy equivalence
$$\S K_{n}\simeq S^2\vee\bigvee_{\text{even }i>0}\tbinom {n-1}i(S^{i+1}\vee S^{i+2})\vee\bigvee_{\text{odd }i}\tbinom {n-1}i M^{i+2}(2),$$
where $M^{i+2}(2)$ denotes the mod-2 Moore space $S^{i+1}\cup_2 e^{i+2}$.\end{thm}

\section{$K_n$ as a planar polygon space}\label{TCsec}
In this section, we explain how the spaces $K_n$ initially came to our attention as planar polygon spaces, and what we might hope to discover about their topological complexity. Recall that $\Mbar(\ell)=\Mbar(\ell_1,\ldots,\ell_n)$ is the space of planar polygons with side lengths $\ell_1,\ldots,\ell_n$, identified under isometry. If $\ell$ is generic, then $\Mbar(\ell)$ is an
$(n-3)$-manifold, and hence satisfies $\TC(\Mbar(\ell))\le 2n-5$.

We may assume $\ell_1\le\cdots\le\ell_n$.  A subset of $[\![n]\!]$ is {\it short} if $\ds\sum_{i\in S}\ell_i<\sum_{i\not\in S}\ell_i$. The {\it genetic code} of $\ell$ is the set of maximal (under an ordering of subsets based on inclusion of sets and size of numbers) elements (called {\it genes}) in the set of short subsets of $[\![n]\!]$ which contain $n$. The homeomorphism type of $\Mbar(\ell)$ is determined by its genetic code. A {\it gee} is a gene with the $n$ omitted.

The polygon space $\Mbar(1^{n-1},n-2)$ is homeomorphic to $RP^{n-3}$(see \cite{DP}), whose topological complexity is usually 1 greater than its immersion dimension(see \cite{FTY}), and this is known to often be much less than $2n-6$. Its genetic code is $\la\{n\}\ra$, so its set of gees is $\{\emptyset\}$. The polygon space $\Mbar(0^{n-3},1,1,1)$ is homeomorphic to $T^{n-3}$(see \cite{geo}). This uses the convention that 0-lengths represent edges the sum of whose lengths is less than 1. Its topological complexity is $n-2$(see \cite{F}), and its genetic code is $\la\{n,n-3,n-4,\ldots,1\}\ra$.

In \cite{D2}, it is proved that for any set of gees except $\{\emptyset\}$, the associated set of $n$-gons has topological complexity $\ge 2n-6$ for sufficiently large $n$, by exhibiting
elements $x_1,\ldots,x_{2n-7}\in H^1(\Mbar(\ell);\zt)$ such that \begin{equation}\label{zcl}\Prod_{i=1}^{2n-7}(x_i\ot 1+1\ot x_i)\ne0\in H^*(\Mbar(\ell)\times \Mbar(\ell);\zt).\end{equation} It is also shown that for $n\le8$, excluding the $RP^{n-3}$ and $T^{n-3}$ cases discussed above, the only genetic codes for which we cannot find classes satisfying (\ref{zcl}) are those with a single gene $\{7,3,2,1\}$, $\{7,5,2,1\}$, $\{8,4,3,2,1\}$, or $\{8,6,3,2,1\}$.  There are more than 2600 genetic codes with $n\le 8$.

The genetic code with single gene $\{n,n-4,n-5,\ldots,1\}$ is closest to that of the torus, and would seem to be the best candidate to have topological complexity less than $2n-6$. It is realized by the length vector $(0^{n-4},1,1,1,2)$. By \cite[Prop 2.1]{geo}, this space is homeomorphic to the space $K_{n-3}$ defined by (\ref{def1}).  We will show in Proposition \ref{zclprop}
that for $K_{n-3}$ the largest nonzero product of the form (\ref{zcl}) has $n-1$ factors, and so all we can deduce is $\TC(K_{n-3})\ge n$. The original goal of this project was to try to decrease the gap ($n$ to $2n-5$) for $\TC(K_{n-3})$.

Here is the result  that is useful in obtaining lower bounds for topological complexity. It is convenient to denote $V\ot1+1\ot V$ by $\Vbar$.
\begin{prop}\label{zclprop} In $H^*(K_n\times K_n;\zt)$, $\Vbar_1^3\Vbar_2^2\Vbar_3\cdots \Vbar_{n-1}\ne0$, but any product of at least $n+3$ terms of the form $\Vbar_i$ or $\Rbar$ is 0.
\end{prop}
\begin{proof} The expansion of the stated product includes the term
$$V_1V_2^2V_3\cdots V_{n-2}\ot V_1^2V_{n-1}=RV_1\cdots V_{n-2}\ot RV_1V_{n-1}$$
and no other terms that can cancel it.

Let $P_i$ denote a product of $i$ distinct $V$'s.
 In the expansion of the product of $n+k$ $\Vbar_i$'s (some repeated) in bidegree $(d,n+k-d)$, a term $RP_{d-1}\ot RP_{n+k-d-1}$ has $n+k-2$ $V_i$'s occurring, so at least $k-1$ must appear on both sides of the $\ot$. These can only be obtained from $\Vbar^j$ for $j$ not a 2-power, hence $j\ge3$. Such a term has degree $\ge 3(k-1)+(n-1-(k-1))=n+2k-3$.
Therefore $n+k\ge n+2k-3$, so $k\le 3$. Using $P_d$ instead of $RP_{d-1}$, or including an $\Rbar$ factor would similarly imply the stronger result $k\le 2$.
If $k=3$, it would have to be (after possible reindexing)
\begin{eqnarray*}\Vbar_1^3\Vbar_2^3\Vbar_3\cdots\Vbar_{n-1}&=&V_1^2V_2V_3\cdots V_{d-1}\ot V_1V_2^2V_d\cdots V_n\\
&&+V_1V_2^2V_3\cdots V_{d-1}\ot V_1^2V_2V_d\cdots V_n\\
&=&0.\end{eqnarray*}
\end{proof}

One might hope to improve the lower bound using cohomology with local coefficients, as was done for $K_2$ in \cite{CV}. Ordinary integral cohomology won't help since $H^1(K_n\times K_n)$ is spanned by $R\otimes 1$ and $1\otimes R$. On the other hand, one might hope to improve the upper bound by finding motion planning rules similar to those used for the torus, where in each factor we follow the geodesic if the points are not antipodal and move counterclockwise if they are. The domains of continuity for this algorithm are the sets of pairs of points with a fixed number of antipodal components. So far, we have not been able to obtain an improvement of either type.

 \def\line{\rule{.6in}{.6pt}}

\end{document}